\documentclass[12pt,reqno]{amsart}

\usepackage{color}
\usepackage{amsmath, amsthm, amssymb}
\usepackage{amsfonts}

\usepackage[ansinew]{inputenc}
\usepackage[dvips]{epsfig}
\usepackage{graphicx}
\usepackage[english]{babel}
\usepackage{hyperref}
\theoremstyle{plain}
\newtheorem{thm}{Theorem}[section]
\newtheorem{cor}[thm]{Corollary}
\newtheorem{lem}[thm]{Lemma}
\newtheorem{prop}[thm]{Proposition}

\theoremstyle{definition}
\newtheorem{defi}[thm]{Definition}

\theoremstyle{remark}
\newtheorem{rem}[thm]{Remark}

\numberwithin{equation}{section}

\newcommand{\average}{{\mathchoice {\kern1ex\vcenter{\hrule height.4pt
width 6pt depth0pt} \kern-9.7pt} {\kern1ex\vcenter{\hrule
height.4pt width 4.3pt depth0pt} \kern-7pt} {} {} }}

\def\R{\mathbb{R}}

\begin{document}

\title{Pohozaev identities for anisotropic integro-differential operators}

\author{Xavier Ros-Oton}
\address{The University of Texas at Austin, Department of Mathematics, 2515 Speedway, Austin, TX 78751, USA}
\email{ros.oton@math.utexas.edu}

\author{Joaquim Serra}
\address{Universitat Polit\`ecnica de Catalunya, Departament de Matem\`{a}tica  Aplicada I, Diagonal 647, 08028 Barcelona, Spain}
\email{joaquim.serra@upc.edu}

\author{Enrico Valdinoci}
\address{Weierstrass Institut f\"ur Angewandte Analysis und Stochastik, Mohrenstrasse 39, 10117 Berlin, Germany}
\email{enrico.valdinoci@wias-berlin.de}

\thanks{XR and JS were supported by grant MTM2014-52402-C3-1-P}


\maketitle

\begin{abstract}
We find and prove new Pohozaev identities and integration by parts type formulas for anisotropic integro-differential operators of order $2s$, with $s\in(0,1)$.

These identities involve local boundary terms, in which the quantity $u/d^s|_{\partial\Omega}$ plays the role that $\partial u/\partial\nu$ plays in the second order case.
Here, $u$ is any solution to $Lu=f(x,u)$ in $\Omega$, with $u=0$ in $\R^n\setminus\Omega$, and $d$ is the distance to $\partial\Omega$.
\end{abstract}

\section{Introduction and results}

Integro-differential equations arise naturally in the study of stochastic processes with jumps, and more precisely of L\'evy processes.
In the context of L\'evy processes, these equations play the same role that second order PDEs play in the theory of Brownian motions.
This is because infinitesimal generators of  L\'evy processes are integro-differential operators.

A very special class of L\'evy processes is the one corresponding to stable processes.
These are the processes that satisfy certain scaling properties, and in particular they satisfy that the sum of two i.i.d. stable processes is also stable.
The infinitesimal generator of any symmetric stable L\'evy process is of the form
\begin{equation}\label{L-singular}
Lu(x)=\int_{S^{n-1}}\int_{-\infty}^{+\infty}\bigl(2u(x)-u(x+\theta r)-u(x-\theta r)\bigr)\frac{dr}{|r|^{1+2s}}\,d\mu(\theta),
\end{equation}
where $\mu$ is any finite measure on the unit sphere, called the \emph{spectral measure}, and $s\in(0,1)$; see \cite{ST,Levy,Nolan}.

When this measure is absolutely continuous with respect to the classical measure on the sphere, then it can be written as
\begin{equation}\label{L}
Lu(x)= \int_{\R^n}\bigl(2u(x)-u(x+y)-u(x-y)\bigr)\frac{a\left(y/|y|\right)}{|y|^{n+2s}}dy,
\end{equation}
where $a\in L^1(S^{n-1})$ is nonnegative and even ---i.e., $a(\theta)=a(-\theta)$.

As said before, integro-differential equations appear naturally when studying L\'evy processes.
For example, the solution $u(x)$ to the Dirichlet problem in a domain $\Omega$ gives the expected cost of a random motion starting at point $x\in\Omega$, the running cost being the right hand side of the equation.
When this right hand side is $f\equiv1$ in~$\Omega$, then the solution $u(x)$ is the expected first time at which the particle exits the domain.

Linear and nonlinear equations involving this type of operators have been widely studied, from the point of view of both Probability and Analysis; see \cite{Bass2,Bogdan1,CS,FKV,Grubb,K,KM,PT,RS-Dir} for example.

Here we study integro-differential problems of the form
\begin{equation}\label{eq}
\left\{ \begin{array}{rcll}
L u &=&f(x,u)&\textrm{in }\Omega \\
u&=&0&\textrm{in }\R^n\backslash\Omega,
\end{array}\right.
\end{equation}
where $\Omega\subset\R^n$ is a bounded domain, and $L$ is given by either~\eqref{L} or \eqref{L-singular}.

In this paper, we find and prove new Pohozaev-type identities for solutions to \eqref{eq}.

Pohozaev-type identities have been widely used in the theory of PDEs.
In elliptic equations these identities are used to prove sharp nonexistence results, partial regularity of solutions, concentration phenomena, unique continuation properties, or rigidity results \cite{P,R,DS,Han,V,Weinberger}.
Moreover, they are also frequently used in hyperbolic equations, control theory, harmonic maps, and geometry \cite{BO,Strauss,Caz,CZ,Sc,KW,Pol}.

For integro-differential equations, the first identity of this type was established in \cite{RS-Poh}, where
the Pohozaev identity for the fractional Laplacian was proved.
Here, we extend the method introduced in \cite{RS-Poh} to establish Pohozaev-type identities for more general operators of the form \eqref{L} and \eqref{L-singular}.
As explained below, new ideas are required to treat the anisotropic case, in which we obtain the extra factor in the boundary term.

We recall that, for second order equations, Pohozaev-type identities usually follow from the divergence theorem or from the integration by parts formula.
However, for integro-differential equations these tools are not available, and thus the approach to these identities must be completely different.

\subsection{Assumptions}

In order to ensure the regularity of solutions to \eqref{eq}, one has to impose some ellipticity assumptions on the measure $\mu$.
When $L$ is of the form \eqref{L} we will assume that
\begin{equation}\label{ellipt}
0<\lambda\leq \int_{S^{n-1}}a(\sigma)d\sigma,\qquad 0\leq a(\theta)\leq \Lambda<\infty\quad\textrm{for all}\ \theta\in S^{n-1},
\end{equation}
while when $L$ is of the form \eqref{L-singular} we will assume
\begin{equation}\label{ellipt-singular}
0<\lambda\leq \inf_{\nu\in S^{n-1}}\int_{S^{n-1}}|\nu\cdot\sigma|^{2s}d\mu(\sigma),\qquad  \int_{S^{n-1}}d\mu\leq \Lambda<\infty.
\end{equation}

Moreover, in our results we will assume that either
\begin{equation}\label{A1}
L\ \textrm{is of the form \eqref{L}-\eqref{ellipt},}\ \textrm{and}\ \Omega\ \textrm{is}\ C^{1,1};
\end{equation}
or
\begin{equation}\label{A2}
L\ \textrm{is of the form \eqref{L-singular}-\eqref{ellipt-singular},}\ \textrm{and}\ \Omega\ \textrm{is}\ convex\ \textrm{and}\ C^{1,1}.
\end{equation}

The convexity of the domain $\Omega$ in \eqref{A2} is needed in order to ensure certain interior regularity of solutions to \eqref{eq}, as explained later on in this Introduction.

\subsection{Main results}

The following is our main result.

\begin{thm}\label{thpoh}
Let $s\in(0,1)$, and assume that $L$ and $\Omega$ satisfy either \eqref{A1} or \eqref{A2}.

Let $f$ be any locally Lipschitz function, $u$ be any bounded solution to \eqref{eq}, and $d(x)={\rm dist}(x,\R^n\setminus\Omega)$.

Then,
\[u/d^s|_{\Omega}\in C^{\gamma}(\overline\Omega)\quad \textrm{for}\ \gamma<s,\qquad |\nabla u|\leq Cd^{s-1}\ \, \textrm{in}\ \,\Omega,\]
and the following identity holds
\begin{equation}\label{Poh1}
\int_\Omega\biggl\{ (x\cdot\nabla u)Lu+\frac{n-2s}{2}\,u\,Lu\biggl\}\,dx=-\frac{\Gamma(1+s)^2}{2}\int_{\partial\Omega}\mathcal A(\nu)\left(\frac{u}{d^s}\right)^2(x\cdot\nu)d\sigma.
\end{equation}
Moreover, for all $e\in \R^n$, we have
\begin{equation}\label{Poh2}
\int_\Omega \partial_eu\,Lu\,dx=-\frac{\Gamma(1+s)^2}{2}\int_{\partial\Omega}\mathcal A(\nu)\left(\frac{u}{d^{s}}\right)^2(\nu\cdot e)\,d\sigma.
\end{equation}
Here, $\nu$ is the unit outward normal to $\partial\Omega$ at $x$,
\begin{equation}\label{A}
\mathcal A(\nu)=c_s\int_{S^{n-1}}|\nu\cdot\theta|^{2s}a(\theta)d\theta,
\end{equation}
and $c_s$ is a constant that depends only on $s$.
\end{thm}

It is important to notice that $\mathcal A(\xi)$ is the Fourier symbol of the operator \eqref{L}.
(In probability, it is usually called the characteristic exponent of the L\'evy process.)
We think it is an interesting fact that the Fourier symbol $\mathcal A$ appears in these new identities, even though nothing is stated in terms of frequency variables.

In this direction, G. Grubb has recently found a new proof of our Pohozaev identities; see \cite{Grubb-Poh}.
The proofs of Grubb use Fourier transform methods, and are completely different from ours presented above.
A key ingredient in those proofs is the existence of an appropriate factorization of the principal symbol of the operator, in the spirit of her previous works \cite{Grubb,Grubb2}.

The proofs in \cite{Grubb-Poh} are done by flattening the boundary of $\Omega$, and apply as well to $x$-dependent operators.
We think it is important to remark that, even in the case of flat boundary, with these Fourier transform methods the treatment of anisotropic operators $L$ is significantly more delicate than the case $L=(-\Delta)^s$.

The constant $c_s$ in \eqref{A} is given by
\[c_s=\frac{\pi}{\sin(\pi s)\Gamma(1+2s)}.\]
This can be checked by recalling that when $L=(-\Delta)^s$ then $\mathcal A\equiv1$; see~\cite{RS-Poh}.

We recall that the first identity of this type (with a local boundary term) was established by the first two authors in \cite{RS-Poh} in the case of the isotropic fractional Laplacian.
Later, N. Abatangelo \cite{Abatangelo} obtained very related identities involving ``large solutions'' for the fractional Laplacian $(-\Delta)^s$, i.e., solutions that blow up at the boundary of the domain.

As said before, problems of the form \eqref{eq} have a clear probabilistic interpretation, in which $f(x,u)$ can be viewed as a running cost.
Informally speaking, $u(x)$ is the expected cost for a particle that moves randomly, following a L\'evy process starting at $x\in\Omega$.
However, we do not know any probabilistic interpretation of our identities.

As a consequence of Theorem \ref{thpoh} we have the following.

\begin{cor}\label{corpoh}
Let $s\in(0,1)$, and assume that $L$ and $\Omega$ satisfy either \eqref{A1} or \eqref{A2}.
Let $f$ be a locally Lipschitz function, and $u$ be any bounded solution of
\begin{equation}\label{eq2}
\left\{ \begin{array}{rcll}
L u &=&f(u)&\textrm{in }\Omega \\
u&=&0&\textrm{in }\R^n\backslash\Omega,
\end{array}\right.
\end{equation}
Then, the following identity holds
\[\int_{\Omega}\biggl\{2nF(u)-(n-2s)u\,f(u)\biggr\}dx=\Gamma(1+s)^2\int_{\partial\Omega}\mathcal A(\nu)\left(\frac{u}{d^s}\right)^2(x\cdot\nu)d\sigma,\]
where $F(t)=\int_0^tf$, $\nu$ is the unit outward normal to $\partial\Omega$ at $x$, and $\mathcal A$ is given by \eqref{A}.
\end{cor}

Note that the quantity $u/d^s|_{\partial\Omega}$ plays the role that the normal derivative plays in second order PDEs.
This fact is also observed in the Serrin's problem for the fractional Laplacian \cite{DG-V,FJ,soave}.

A consequence of Corollary \ref{corpoh} is the nonexistence of positive solutions to \eqref{eq2}-\eqref{A1} in star-shaped domains
for the critical nonlinearity $f(u)=u^{\frac{n+2s}{n-2s}}$; see Section~\ref{sec9}.

Finally, as in \cite{RS-Poh}, another consequence of Theorem \ref{thpoh} is the following integration by parts formula.

\begin{cor}\label{corintparts}
Let $s\in(0,1)$, and assume that $L$ and $\Omega$ satisfy either \eqref{A1} or \eqref{A2}.

Let $u$ and $v$ be two functions satisfying the hypotheses of Theorem \ref{thpoh} -- with possibly different nonlinearities $f(x,u)$ and $g(x,v)$.

Then, the following identity holds for $i=1,...,n$
\[\int_\Omega Lu\ v_{x_i}\,dx=-\int_\Omega u_{x_i}\,Lv\,dx-\Gamma(1+s)^2\int_{\partial\Omega}\mathcal A(\nu)\frac{u}{d^{s}}\frac{v}{d^{s}}\nu_i\,d\sigma.\]
Here, $\nu$ is the unit outward normal to $\partial\Omega$ at $x$, and $\mathcal A$ is given by \eqref{A}.
\end{cor}

To establish Theorem \ref{thpoh} we have to extend the method in \cite{RS-Poh} for the fractional Laplacian to more general operators \eqref{L}.
In the case $L=(-\Delta)^s$ an important ingredient of the proof in \cite{RS-Poh} was the precise behavior of $(-\Delta)^{s/2}u(x)$ for $x$ near~$\partial\Omega$.

Here, we consider the operator $L^{1/2}$ and we study the singular behavior of the function $L^{1/2}u$ near $\partial\Omega$.
This requires very fine regularity estimates for $u$, $u/d^s$, and $L^{1/2}u(x)$ near the boundary.
Some of these estimates were already established in \cite{RS-stable} and \cite{RV}, while some other estimates are developed in the present paper.

The kernel of the operator $L^{1/2}$ has an explicit expression in case $L=(-\Delta)^s$, but not for general stable operators.
Because of this, the proofs of our Pohozaev identities are simpler for $L=(-\Delta)^s$, and new ideas are required to treat the general case, in which we obtain the extra factor $\sqrt{\mathcal A(\nu(z))}$.

\subsection{Some ingredients of the proof}

As said above, the proof of Theorem \ref{thpoh} follows the strategy in \cite{RS-Poh}.
However, the extension from $(-\Delta)^s$ to more general nonlocal operators \eqref{L} requires new ideas and presents some
interesting mathematical questions, as explained in more detail at the end of this Introduction.

An important ingredient in the proofs of the present paper is the regularity up to the boundary of the quotient $u/d^s$, stated next.
When the spectral measure $a$ and the domain $\Omega$ are $C^\infty$, this is a particular case of the results of G. Grubb \cite{Grubb,Grubb2}.
For non-regular spectral measures and $C^{1,1}$ domains, the regularity of $u/d^s$ was recently established in \cite{RS-stable}.
This result reads as follows.

\begin{thm}[\cite{RS-stable}]\label{krylov}
Let $\Omega$ be any bounded and $C^{1,1}$ domain.
Let $L$ be any operator of the form \eqref{L-singular}-\eqref{ellipt-singular}, and $u\in H^s(\R^n)$ be the solution of $Lu=g$ in $\Omega$, $u=0$ in $\R^n\setminus\Omega$, with $g\in L^\infty(\Omega)$.

Then, $u/d^s$ is H\"older continuous up to the boundary $\partial\Omega$, and
\[\|u/d^s\|_{C^{\gamma}(\overline\Omega)}\le C\|g\|_{L^{\infty}(\Omega)}\qquad \textrm{for all}\ \  \gamma<s.\]
The constant $C$ depends only on $\Omega$, $s$, $\gamma$, and the ellipticity constants.
\end{thm}

Recall that for more general integro-differential operators of order $2s$, solutions $u$ may not be comparable to $d^s$ near the boundary of $\Omega$.
For example, it is
shown in \cite{RS-K} that fully nonlinear equations with respect to the class $\mathcal L_0$ (or even to $\mathcal L_1$ and $\mathcal L_2$) fail to have this property; see Section 2 in \cite{RS-K} for more details.

We will also need the following result, established recently in \cite{RV}, and which deals with the interior regularity of solutions.

\begin{thm}[\cite{RV}]\label{thm-RV}
Let $L$ and $\Omega$ satisfy either \eqref{A1} or \eqref{A2}.
Let $u\in H^s(\R^n)$ be the solution of $Lu=g$ in $\Omega$, $u=0$ in $\R^n\setminus\Omega$.
Assume that $g\in L^\infty(\Omega)$ and that $|\nabla g|\leq Cd^{-s-1}$ in $\Omega$.

Then, $u$ is $C^{1+2s-\epsilon}_{\rm loc}(\Omega)$ for all $\epsilon>0$, with the estimate
\[[u]_{C^{s+\beta}(\{{\rm dist}(x,\partial\Omega)>\rho\})}\le C \rho^{-\beta}\qquad \textrm{for all}\ \  \rho\in(0,1),\]
for all $\beta\in[0,1+s)$.
\end{thm}

Moreover, we showed in \cite{RV} that there exists a \emph{nonconvex} $C^\infty$ domain and an operator \eqref{L-singular}-\eqref{ellipt-singular} for which the solution of \eqref{eq} with $f\equiv1$ is \emph{not} $C^{0,1}_{\rm loc}(\Omega)$.
In particular, and somewhat surprisingly, the statement of Theorem~\ref{thm-RV} becomes false when both conditions~\eqref{A1} and~\eqref{A2} are dropped.
This is the essential reason for which we assume~\eqref{A1} or~\eqref{A2} in the present paper.

\begin{rem}
The ellipticity assumption in \eqref{ellipt} looks at first glance different from the one in \cite{RS-stable,RV} (which is the one in \eqref{ellipt-singular}).
However, for spectral functions $a\in L^\infty(S^{n-1})$ these two ellipticity assumptions are equivalent, and hence we can apply the results of \cite{RS-stable} and \cite{RV}.
\end{rem}

In our setting, Theorem \ref{thpoh} will follow from Proposition \ref{intparts} below.

\begin{prop}\label{intparts}
Let $L$ and $\Omega$ satisfy either \eqref{A1} or \eqref{A2}.
Let $u\in H^s(\R^n)$ be the solution of $Lu=g$ in $\Omega$, $u=0$ in $\R^n\setminus\Omega$.
Assume that $g\in L^\infty(\Omega)$, and that $|\nabla g|\leq Cd^{-s-1}$ in $\Omega$.

Then, $u/d^s$ is H\"older continuous up to the boundary, $|\nabla u|\leq Cd^{s-1}$ in $\Omega$, and the following identity holds
\[\int_\Omega(x\cdot\nabla u)Lu\ dx=\frac{2s-n}{2}\int_{\Omega}u\,Lu\, dx-\frac{c_s}{2}\int_{\partial\Omega}\mathcal A(\nu)\left(\frac{u}{d^s}\right)^2(x\cdot\nu)d\sigma.\]
Here, $\nu$ is the unit outward normal to $\partial\Omega$ at $x$, and $\mathcal A$ is given by \eqref{A}.
\end{prop}

The hypotheses of this Proposition will be satisfied for any solution to the semilinear elliptic equation \eqref{eq}.
Still, we expect solutions to other related equations, like $u_t+Lu=f(x,u)$, to satisfy the same hypotheses; see \cite{FR}.

\vspace{2mm}

The paper is organized as follows.
In Section \ref{approx} we show that it suffices to prove Proposition \ref{intparts} for $C^\infty$ spectral measures.
In Section \ref{Fourier} we give a description of the operator $L^{1/2}$.
In Section \ref{regv} we prove some interior regularity results for the quotient $u/d^s$, which are important in our proof of Proposition \ref{intparts}.
Then, in Section~\ref{secs/2} we study the singular behavior of the function $L^{1/2}u$ near the boundary $\partial\Omega$.
In Section~\ref{sec2} we give the proof of Proposition \ref{intparts} in the case of star-shaped domains.
In Section~\ref{sec8} we finish the proof of Proposition \ref{intparts} and we prove Theorem \ref{thpoh}.
Finally, in Section \ref{sec9} we give some applications of our identities.

Let us stress the main novelties of the present paper with respect to the results in \cite{RS-Poh}.
The contents of Sections 2 and 3 are new with respect to \cite{RS-Poh}, while the results of Section 4 are a modified (and simplified)
version of the corresponding ones in \cite{RS-Dir}.
The results in Sections 5 and 6 have been carefully adapted to the present case of anisotropic operators.
The analysis of these two Sections is more delicate than the one for $L=(-\Delta)^s$, and it is here where the new factor $\mathcal A(\nu)$ shows up in the boundary term of the identity.
The proofs of Section 7 are more similar to \cite{RS-Poh}.
Finally, the results in Section 8 are new even for the fractional Laplacian.

Throughout the paper we will skip the parts of the proofs that are more similar to the ones in \cite{RS-Poh}, to focus in the ones that present new mathematical ideas and/or difficulties.

\section{An approximation argument}
\label{approx}

The hypotheses of Proposition \ref{intparts} allow the spectral measures $a(\cdot)$ to be very irregular.
In this section we show that, by an approximation argument, it suffices to consider the case in which $a\in C^\infty(S^{n-1})$.

More precisely, in this Section we assume that the following result holds, and we prove that Proposition \ref{intparts} follows from it.

\begin{prop}\label{intpartsB}
Let $\Omega$ be any $C^{1,1}$ domain, and let $L$ be an operator of the form \eqref{A1}, with $a\in C^\infty(S^{n-1})$.
Let $u\in H^s(\R^n)$ be any function satisfying
\begin{itemize}
\item[(a)] $u=0$ in $\R^n\setminus\Omega$.
\item[(b)] For all $\beta\in[0,1+s)$ and all $\rho>0$, we have
\[[u]_{C^{s+\beta}(\{{\rm dist}(x,\partial\Omega)>\rho\})}\leq C\rho^{-\beta}.\]
\item[(c)] $Lu$ is bounded in $\Omega$.
\end{itemize}

Then, $u/d^s$ is H\"older continuous up to the boundary, and the identity \eqref{Poh1} holds.
\end{prop}

Let us give next the proof of Proposition \ref{intparts}.
After this, the rest of the paper will consist essentially
on the proof of Proposition \ref{intpartsB}
(the proof of Proposition~\ref{intpartsB} will be completed
on Section~\ref{sec8} and this will at once also
give the proof of Proposition~\ref{intparts} and Theorem \ref{thpoh}).

\begin{proof}[Proof of Proposition \ref{intparts}]
Let $\Omega$ and $L$ satisfy either \eqref{A1} or \eqref{A2}, and let $u$ and $g$ be as in the statement of Proposition \ref{intparts}.

Let $a_k\in C^\infty(S^{n-1})$ be a sequence of nonnegative functions converging weakly towards the spectral measure of the operator $L$.
Let $L_k$ be the operator \eqref{L} whose spectral measure is $a_k$, and let $u_k$ be the solution of
\[\left\{ \begin{array}{rcll}
L_k u_k &=&g&\textrm{in }\Omega \\
u_k&=&0&\textrm{in }\R^n\backslash\Omega.
\end{array}\right.\]
Then, by Theorems \ref{krylov} and \ref{thm-RV}, we have
\[\|u_k\|_{C^s(\R^n)}\leq C,\qquad \|u_k/d^s\|_{C^\gamma(\overline\Omega)}\leq C,\qquad |\nabla u_k|\leq Cd^{s-1},\]
for some constant $C$ that depends on $g$, $n$, $\Omega$, and the ellipticity constants, but not on $k$.

Thus, up to a subsequence, the sequence $u_k$ converges uniformly to a function $w$ which satisfies $w\equiv0$ in $\R^n\setminus\Omega$,
\[\|w\|_{C^s(\R^n)}\leq C,\qquad \|w\|_{C^\gamma(\overline\Omega)}\leq C,\qquad |\nabla w|\leq Cd^{s-1}.\]
Furthermore, since the functions $u_k$ satisfy
\[[u_k]_{C^{s+\beta}(\{{\rm dist}(x,\partial\Omega)>\rho\})}\le C \rho^{-\beta}\qquad \textrm{for all}\ \  \rho\in(0,1),\]
for all $\beta\in[0,1+s)$, then the same bound holds for the function $w$.

This allows us to show that, for every $x\in \Omega$, $L_ku_k$ is defined pointwise, and
\[g(x)=L_ku_k(x)\longrightarrow Lw(x).\]
Hence, $Lw=g$ in $\Omega$.

But then, by uniqueness of the solution to $Lu=g$ in $\Omega$, $u=0$ in $\R^n$, we have that $u\equiv w$.

Finally, since each $u_k$ satisfy the hypotheses of Proposition \ref{intpartsB}, then we have that
\[\int_\Omega(x\cdot\nabla u_k)g\ dx=\frac{2s-n}{2}\int_{\Omega}u_k\,g\, dx-\frac{c_s}{2}\int_{\partial\Omega}\mathcal A(\nu)\left(\frac{u_k}{d^s}\right)^2(x\cdot\nu)d\sigma.\]
Thus, taking the limit $k\rightarrow\infty$ in the previous identity, we find \eqref{Poh1}, and thus we are done.
\end{proof}

\section{Fourier symbols and kernels}
\label{Fourier}

The proof of the Pohozaev identity \eqref{Poh1} follows the steps of the one for the fractional Laplacian $(-\Delta)^s$ in \cite{RS-Poh}.
In the proof of \cite{RS-Poh}, the function $(-\Delta)^{s/2}u$ played a very important role, and this role will be played here by the $L^{1/2}u$.

In order to establish fine estimates for this function $L^{1/2}u$, we will need the following result,
which states that the square root of~$L$ also possesses an associated
spectral measure.

\begin{lem}\label{lem-Fourier}
Let $s\in (0,1)$, and $L$ be an operator of the form \eqref{L}-\eqref{ellipt}, with $a\in C^\infty(S^{n-1})$.
Then, there exists $b\in C^\infty(S^{n-1})$ such that
\[L^{1/2}u(x)=\int_{\R^n}\bigl(u(x)-u(x+y)\bigr)\frac{b(y/|y|)}{|y|^{n+s}}\,dy.\]
Moreover, the function $b$ satisfies
\begin{equation}\label{B-symbol}
\int_{S^{n-1}}|\nu\cdot\theta|^s b(\theta)d\theta=c\left(\int_{S^{n-1}}|\nu\cdot\theta|^{2s}a(\theta)d\theta\right)^{1/2}
\end{equation}
for all $\nu\in S^{n-1}$, for some constant $c$.
\end{lem}

\begin{proof}
The Fourier symbol of $L$ is given by
\[ {\mathcal A}(\xi)=c\int_{S^{n-1}}|\xi\cdot\theta|^{2s}a(\theta)d\theta;\]
see for example \cite{ST}.
Thus, the Fourier symbol of $L^{1/2}$ is given by
\[ {\mathcal B}(\xi)=\left(c\int_{S^{n-1}}|\xi\cdot\theta|^{2s}a(\theta)d\theta\right)^{1/2}.\]
This symbol is homogeneous of degree $s$, and is positive and $C^\infty$ in $\R^n\setminus\{0\}$.
Hence, this means that the operator can be written as
\[L^{1/2}w(x)=\int_{\R^n}\bigl(u(x)-u(x+y)\bigr)K(y)dy,\]
for some kernel $K(y)$ homogeneous of degree $-n-s$, and such that $K\in C^\infty(\R^n\setminus\{0\})$; see for example Section~0.2 in \cite{TaylorBook}.

In other words, we may write $K$ as
\[K(y)=\frac{b(y/|y|)}{|y|^{n+s}},\]
with $b\in C^\infty(S^{n-1})$, as desired.

In fact, the function $b$ can be computed explicitly in terms of $\mathcal B$ by using that, for any $\alpha\in\mathbb N_\circ^n$ with $|\alpha|=n$, we have
\[y^\alpha K(y)=c\int_{S^{n-1}}|y\cdot\theta|^{-s} D^\alpha \mathcal B(\theta)d\theta.\]
for all $y\in \R^n$.

It is important to notice that since $\mathcal B$ is even then $b$ will be even, but that the positivity of $\mathcal B$ does not yield the positivity of $b$.
\end{proof}

\begin{rem}
We expect a similar result to hold not only for spectral measures $a\in C^\infty(S^{n-1})$, but also for $a\in L^\infty(S^{n-1})$ or for general measures $\mu$.
However, we do not need this here, since by the approximation argument in the previous Section we can assume from now on that $a\in C^\infty(S^{n-1})$.
\end{rem}

\section{Interior regularity for $u/d^s$}
\label{regv}

In this section we will obtain interior estimates for the quotient $u/d^s$, that is, Proposition \ref{intregquotient} below.
These estimates hold for all operators \eqref{L-singular}-\eqref{ellipt-singular} in any $C^{1,1}$ domain $\Omega$
(with no convexity assumption on the domain, with no regularity assumptions on the spectral measure).

Throughout this section, $L$ is any operator of the form \eqref{L-singular}-\eqref{ellipt-singular}.
Also, throughout this section, $d$ is a $C^{1,1}$ function that coincides with $\textrm{dist}(x,\R^n\setminus\Omega)$ in a neighborhood of $\partial\Omega$.
That is, $d$ is just the distance function but avoiding possible singularities inside $\Omega$.

As in \cite{RS-Dir}, the key idea to obtain these estimates is to use the following equation
\[Lv=\frac{1}{d^s}\bigl\{Lu-v\,Ld^s+I_L(v,d^s)\bigr\}\qquad \textrm{in}\quad \Omega,\]
where $v\in C^\gamma(\R^n)$ is an extension of $u/d^s|_\Omega$, with $\gamma\in(0,s)$, and
\begin{equation}\label{I_L}
I_L(w_1,w_2)=\int_{\R^n}\bigl(w_1(x)-w_1(x+y)\bigr)\bigl(w_2(x)-w_2(x+y)\bigr)\frac{a(y/|y|)}{|y|^{n+2s}}\,dy.
\end{equation}

The following is the main result of this section.

\begin{prop}\label{intregquotient}
Let $L$ and $\Omega$ be as in \eqref{A2}, and $u$ be such that $u\equiv0$ in $\R^n\setminus\Omega$ and $\|Lu\|_{L^\infty(\Omega)}\leq C$.
Then, for all $\gamma<s$ and for all $\beta<2s$
\[ [u/d^s]_{C^{\beta}(\{{\rm dist}(x,\partial\Omega)>\rho\})}\leq C\rho^{\gamma-\beta}\qquad \textrm{for all}\quad \rho\in(0,1),\]
where $C$ is a constant that do not depend on $\rho$.
\end{prop}

The proof of this result is a modified
(and even somehow simplified)
version of the one in \cite[Section 4]{RS-Dir}.

As said before, we need several lemmas to prove Proposition \ref{intregquotient}.
We start with the first one, which reads as follows.

\begin{lem}\label{Lds}
Let $\Omega$ be any $C^{1,1}$ bounded domain, $s\in (0,1)$, $L$ be given by \eqref{L}.
Then, for all $\epsilon>0$ there exists a constant $C$ such that
\[\|d^\epsilon\,Ld^s\|_{L^\infty(\Omega)}\leq C.\]
Moreover, the constant $C$ depends only on $n$, $s$, $\epsilon$, $\Lambda$, and $\Omega$.
\end{lem}

\begin{proof}
Note that $d^s$ is $C^{1,1}$ inside $\Omega$, so we only need to prove that $|d^{\epsilon}(x)Ld^s(x)|\leq C$ for $x\in\Omega$ near $\partial\Omega$.

Let $x\in\Omega$, and let $x_0\in\partial\Omega$ be such that $|x-x_0|=d(x)$.
Let us consider the function $\varphi_{x_0}(x)=(-x\cdot\nu)_+^s$, where $\nu$ is the unit outward normal to $\partial\Omega$ at $x_0$.
It follows from Lemma 2.1 in \cite{RS-K} that
\[L\varphi_{x_0}(x)=0;\]
see Section 2 in \cite{RS-K} for more details.
Hence, we only have to prove that
\[Lw(x)\leq C_0d^{-\epsilon}(x),\]
where we have denoted $w=d^s-\varphi_{x_0}$.

Let $\rho=d(x)/2$.
Then, the function $w$ satisfies
\[|w(x+y)|\leq
\left\{\begin{array}{ll}
C\rho^{s-1}|y|^{2}\ &\textrm{for} \ y\in B_\rho,\\
C|y|^{2s}\ &\textrm{for} \ y\in B_1\setminus B_\rho,\\
C|y|^{s}\ &\textrm{for} \ y\in \R^n\setminus B_1.
\end{array}\right.\]
Therefore, we have that
\[\begin{split}
|Lw(x)|&\leq\int_{\R^n}\bigl|w(x)-w(x+y)\bigr|\frac{\Lambda}{|y|^{n+2s}}\,dy\\
&\leq  \Lambda\int_{B_\rho}\frac{\rho^{s-1}|y|^2}{|y|^{n+2s}}\,dy+
\Lambda\int_{B_1\setminus B_\rho}\frac{|y|^{2s}}{|y|^{n+2s}}\,dy+\Lambda\int_{\R^n\setminus B_1}\frac{|y|^s}{|y|^{n+2s}}\,dy\\
&\leq C\left(1+|\log \rho| \right)\\
&\leq Cd^{-\epsilon}(x),\end{split}\]
as desired.
\end{proof}

The next result is the analog of Corollary 2.5 in \cite{RS-Dir}, and can be found in \cite{RS-stable}.

\begin{lem}[\cite{RS-stable}]\label{2s-gain-1}
Let $L$ be given by \eqref{L}, and let $w\in C^\infty(\R^n)$.
Then, for all $\beta<2s$ and $\epsilon>0$,
\[\|w\|_{C^{\beta}(B_{1/2})}\leq C\left(\|Lw\|_{L^\infty(B_1)}+\|w\|_{L^\infty(B_1)}+\sup_{R\geq1}\left\{R^{\epsilon-2s}\|w\|_{L^\infty(B_R)}\right\}\right),\]
where $C$ is a constant depending only on $n$, $s$, $\beta$, $\epsilon$, $\lambda$, and $\Lambda$.
\end{lem}

As a consequence of the previous lemma we find the following.

\begin{lem} \label{2s-gain}
Let $s$ and $\gamma$ belong to $(0,1)$, with $\gamma<2s$.
Let $U$ be an open set with nonempty boundary.
Then, for all $\beta<2s$,
\[ \|w\|_{\beta;U}^{(-\gamma)}\le C\biggl( \|w\|_{L^\infty(\R^n)}+ \|Lw\|_{0;U}^{(2s-\gamma)}\biggr)\]
for all $w$ with finite right hand side.
The constant $C$ depends only on $n$, $s$, $\gamma$, and $\beta$.
\end{lem}

\begin{proof}
For each $x_0\in U$, let $R=\textrm{dist}(x_0,\partial U)/2$ and $\tilde w(y)=w(x_0+Ry)-w(x_0)$.
Then, we have that
\[\|\tilde w\|_{C^\gamma(B_1)}\leq R^\gamma [w]_{C^\gamma(\R^n)},\]
\[\sup_{\rho\geq1}\rho^{-\gamma}\|\tilde w\|_{L^\infty(B_\rho)}\leq R^\gamma[w]_{C^\gamma(\R^n)},\]
and
\[\|L\tilde w\|_{L^\infty(B_1)}=R^{2s}\|Lw\|_{L^\infty(B_R(x_0))}\leq R^\gamma \|Lw\|_{0;U}^{(2s-\gamma)}.\]
Hence, using Lemma \ref{2s-gain-1}, we find that
\[\|\tilde w\|_{C^\beta(B_{1/2})} \leq CR^\gamma\left([w]_{C^\gamma(\R^n)}+\|Lw\|_{0;U}^{(2s-\gamma)}\right).\]
Then, since this happens for all $x_0\in U$, the proof finishes exactly as in the proof of \cite[Lema 2.10]{RS-Dir}.
\end{proof}

Finally, the last ingredient for the proof of Proposition \ref{intregquotient} is the following.

\begin{lem} \label{bound-I}
Let $\Omega$ be a bounded $C^{1,1}$ domain, and $U\subset \Omega$ be an open set.
Let $s$ and $\epsilon$ belong to $(0,1)$ and satisfy $\epsilon<s$.
Then,
\begin{equation}\label{eq:bound-I}
\|I_L(w,d^s)\|_{0;U}^{(s-\epsilon)}\le C\biggl( [w]_{C^{\epsilon}(\R^n)}+ [w]_{\epsilon+s;U}^{(-\epsilon)}\biggr)\,,
\end{equation}
for all $w$ with finite right hand side.
The constant $C$ depends only on $\Omega$, $s$, and $\epsilon$.
\end{lem}

\begin{proof}
Let $x_0\in U$ and $R=\textrm{dist}(x_0,\partial U)/2$.
Let
\[K=\biggl( [w]_{C^{\epsilon}(\R^n)}+ [w]_{\epsilon+s;U}^{(-\epsilon)}\biggr)\biggl( [d^s]_{C^s(\R^n)}+ [d^s]_{\epsilon+s;U}^{(-s)}\biggr).\]
We have that
\[\begin{split}
|I_L(w,d^s)(x_0)|&\leq \Lambda\int_{\R^n}|w(x_0)-w(x_0+y)|\,|d^s(x_0)-d^s(x_0+y)|\frac{dy}{|y|^{n+2s}}\\
&\leq C\int_{B_R(0)}R^{-\epsilon-s}[w]_{\epsilon+s;U}^{(-\epsilon)}[d^s]_{\epsilon+s;U}^{(-s)}|y|^{2\epsilon+2s}\frac{dy}{|y|^{n+2s}}\\
&\qquad \qquad+C\int_{\R^n\setminus B_R(0)} [w]_{C^\epsilon(\R^n)}] [d^s]_{C^s(\R^n)}|y|^{\epsilon+s}\frac{dy}{|y|^{n+2s}}\\
&\leq CR^{\epsilon-s}K.
\end{split}\]
Hence, the result follows.
\end{proof}

We can now continue with the proof of Proposition \ref{intregquotient}.
To complete it, we need to recall the definition of the following weighted H\"older norms:

\begin{defi}\label{definorm} Let $\beta>0$ and $\sigma\ge -\beta$. Let $\beta=k+\beta'$, with $k$ integer and $\beta'\in (0,1]$.
For $w\in C^{\beta}(\Omega)=C^{k,\beta'}(\Omega)$, define the seminorm
\[ [w]_{\beta;\Omega}^{(\sigma)}= \sup_{x,y\in \Omega} \biggl(\min\{d(x),d(y)\}^{\beta+\sigma} \frac{|D^{k}w(x)-D^{k}w(y)|}{|x-y|^{\beta'}}\biggr).\]
For $\sigma\geq0$, we also define  the norm $\|\cdot\|_{\beta;\Omega}^{(\sigma)}$ as follows: in case that $\sigma\ge0$,
\[ \|w\|_{\beta;\Omega}^{(\sigma)} = \sum_{l=0}^k \sup_{x\in \Omega} \biggl(d(x)^{l+\sigma} |D^l w(x)|\biggr) + [w]_{\beta;\Omega}^{(\sigma)}\,,\]
while
\[\|w\|_{\beta;\Omega}^{(-\sigma)} = \|w\|_{C^{\sigma}(\overline \Omega)}+\sum_{l=1}^k \sup_{x\in \Omega} \biggl(d(x)^{l-\sigma} |D^l w(x)|\biggr) + [w]_{\beta;\Omega}^{(-\sigma)}.\]
\end{defi}

\begin{proof}[Proof of Proposition \ref{intregquotient}]
Let $v$ be a $C^\gamma(\R^n)$ extension of $u/d^s|_\Omega$.
Then, as in \cite[Section 4]{RS-Dir}, we have that $v$ solves the equation
\begin{equation}\label{eqnv}
Lv=\frac{1}{d^s}\left\{Lu-v\,Ld^s+I_L(v,d^s)\right\}\qquad \textrm{in}\quad \Omega,
\end{equation}
where
\[I_L(f,g)=\int_{\R^n}\bigl(f(x)-f(x+y)\bigr)\bigl(g(x)-g(x+y)\bigr)\frac{a(y/|y|)}{|y|^{n+2s}}\,dy.\]
Here, $d$ is a function that coincides with $\textrm{dist}(x,\R^n\setminus\Omega)$ in a neighborhood of $\partial\Omega$ and that is $C^{1,1}$ inside $\Omega$.
With this slight modification on the distance function, we will have that \eqref{eqnv} holds everywhere inside $\Omega$.

We want to prove that
\[\|v\|_{\beta;\,\Omega}^{(-\gamma)}\leq C,\]
where the H\"older norms $\|\cdot\|_{\beta}^{(-\gamma)}$ are defined above.

Let us use the equation for $v$ to prove the result.
Let $U\subset\subset\Omega$.
We prove next that
\[\|v\|_{\beta;\,U}^{(-\gamma)}\leq C\]
for some constant $C$ independent of $U$, and this will yield the desired result.

Since $v=u/d^s$ in $\Omega$, and $u\in C^{2s-\epsilon}$ and $d^s\in C^{1,1}$ inside $\Omega$, then it is clear that $\|v\|_{\beta;\,U}^{(-\gamma)}<\infty$.
Next we obtain an a priori bound for this seminorm in $U$.
To do it, we use equation \eqref{eqnv} and Lemma \ref{2s-gain}.
Namely,
\[\begin{split}
\|v\|_{\beta;\,U}^{(-\gamma)} &\leq \|Lv\|_{0;\,U}^{(2s-\gamma)}\\
&\leq \|d^{-s}v\,Ld^s\|_{0;\,U}^{(2s-\gamma)}+\|d^{-s}\,Lu\|_{0;\,U}^{(2s-\gamma)}+\|d^{-s}\, I_L(v,d^s)\|_{0;\,U}^{(2s-\gamma)}.
\end{split}\]

Now, by Lemma \ref{Lds} (with $\epsilon=s-\gamma>0$), we have
\[\|d^{-s}v\,Ld^s\|_{0;\,U}^{(2s-\gamma)}\leq C\|d^{s-\gamma}v\,Ld^s\|_{L^\infty(U)}\leq C\|v\|_{L^\infty(\Omega)}.\]
Similarly,
\[\|d^{-s}\, Lu\|_{0;\,U}^{(2s-\gamma)}\leq C\|Lu\|_{L^\infty(\Omega)}.\]
Moreover, by Lemma \ref{bound-I} (with $\epsilon=s-\gamma$), we have
\[\|I_L(v,d^s)\|_{0;U}^{(s-\gamma)}\le C\biggl( \|v\|_{C^{\gamma}(\R^n)}+ \|v\|_{\gamma+s;U}^{(-\gamma)}\biggr).\]
Thus, assuming $\beta>\gamma+s$ without loss of generality, we deduce that
\[\begin{split}
\|v\|_{\beta;\,U}^{(-\gamma)} &\leq C\biggl(\|Lu\|_{L^\infty(\Omega)}+ \|v\|_{C^{\gamma}(\R^n)}+ \|v\|_{\gamma+s;U}^{(-\gamma)}\biggr)\\
&\leq C\biggl(\|Lu\|_{L^\infty(\Omega)}+ \|v\|_{C^{\gamma}(\R^n)}\biggr)+ \frac12\|v\|_{\beta;U}^{(-\gamma)}.\\
\end{split}\]
This last inequality is by standard interpolation.

Hence, we have proved that
\[\|v\|_{\beta;\,U}^{(-\gamma)} \leq C\biggl(\|Lu\|_{L^\infty(\Omega)}+ \|v\|_{C^{\gamma}(\R^n)}\biggr),\]
and letting $U\uparrow\Omega$ we obtain the desired result.
\end{proof}

\section{Behavior of $L^{1/2}u$ near $\partial\Omega$}
\label{secs/2}

Throughout this section, $L$ is an operator of the form \eqref{L}-\eqref{ellipt} with $a\in C^\infty(S^{n-1})$.

We will also use the following:

\begin{defi}\label{defi-cone}
Given a $C^{1,1}$ domain $\Omega$ a point $x_0\in \partial\Omega$, and $\varepsilon>0$, we define the cone
\[\mathcal C_{x_0}=\{|(x_0-x)\cdot\nu|\geq \varepsilon\,|x-x_0|\},\]
where $\nu=\nu(x_0)$ is the outward unit normal to $\partial\Omega$ at $x_0$.
We also consider
\[\mathcal C_{x_0}^+=\{(x_0-x)\cdot\nu\geq \varepsilon\,|x-x_0|\}\quad\textrm{and}\quad \mathcal C_{x_0}^-=\mathcal C_{x_0}\setminus\mathcal C_{x_0}^+,\]
and a ball $B_\rho(x_0)$, with $\rho>0$ small enough so that $\mathcal C_{x_0}^+\cap B_\rho(x_0)\subset\Omega$ and $\mathcal C_{x_0}^-\cap B_\rho(x_0)\subset \R^n\setminus\Omega$.
\end{defi}

\begin{thm}\label{thlaps/2}
Let $\Omega$ be a bounded and $C^{1,1}$ domain, $L$ be given by \eqref{L}-\eqref{ellipt} with $a\in C^\infty(S^{n-1})$, and $u$ be a function such that $u\equiv0$ in $\R^n\setminus\Omega$ and that $Lu$ is bounded in $\Omega$.
Let $x_0\in \partial\Omega$, and let $\nu$, $\mathcal C_{x_0}$ and $\rho$ as in Definition \ref{defi-cone}.

Then, for all $x\in \mathcal C_{x_0}\cap B_\rho(x_0)$,
\[L^{1/2}u(x)= c_1\bigl\{\log^-|x-x_0|+c_2\chi_{\Omega}(x)\bigr\}\sqrt{\mathcal A(\nu(x_0))}\left(\frac{u}{d^s}\right)(x_0)+h(x),\]
where $\mathcal A$ is given by \eqref{A}, and $h$ is a $C^\gamma(\R^n)$ function satisfying
\[\|h\|_{C^\gamma(\mathcal C_{x_0}\cap B_\rho(x_0))}\leq C,\]
with $C$ independent of $x_0$.

Here, the number $(u/d^s)(x_0)$ has to be understood as a limit (recall that $u/d^s\in C^\alpha(\overline\Omega)$), and $c_1$ and $c_2$ are constants that depend only on $n$ and $s$.
\end{thm}

The proof of this result is split into two results: Propositions \ref{proplaps2} and \ref{prop:Lap-s/2-delta-s}.

The first one, stated next, compares the behavior of $L^{1/2}u$ near $\partial\Omega$ with the one of $L^{1/2}(d^s)$.
Recall that, by Lemma \ref{lem-Fourier},
\[L^{1/2}w(x)=\int_{\R^n}\bigl(w(x)-w(x+y)\bigr)\frac{b(y/|y|)}{|y|^{n+s}}\,dy,\]
for some $b\in C^\infty(S^{n-1})$.

\begin{prop}\label{proplaps2}
Let $\Omega$ be a bounded and $C^{1,1}$ domain, $L$ be given by \eqref{L}-\eqref{ellipt} with $a\in C^\infty(S^{n-1})$, and $u$ be a function such that $u\equiv0$ in $\R^n\setminus\Omega$ and that $Lu$ is bounded in $\Omega$.

Then, there exists a $C^{\alpha}(\R^n)$ extension $v$ of $u/d^s|_\Omega$ such that
\[L^{1/2}u=v\,L^{1/2}d^s+h\ \mbox{ in }\ \R^n,\]
where $h\in C^{\alpha}(\R^n)$, and
\[\|h\|_{C^\alpha(\R^n)}\leq C\]
for some constant $C$ that does not depend on $\theta$.
\end{prop}

The second result gives the singular behavior of $L^{1/2}d^s$ near $\partial\Omega$.

It is important to notice that, in the following result, $d\equiv0$ in $\R^n\setminus\Omega$, while $\delta>0$ in $\R^n\setminus\Omega$.

\begin{prop}\label{prop:Lap-s/2-delta-s}
Let $\Omega$ be a bounded and $C^{1,1}$ domain, $L$ be given by \eqref{L}-\eqref{ellipt} with $a\in C^\infty(S^{n-1})$.
Let $x_0\in \partial\Omega$, and let $\nu$, $\mathcal C_{x_0}$ and $\rho$ as in Definition \ref{defi-cone}.

Then,  for all $x\in \mathcal C_{x_0}\cap B_\rho(x_0)$,
\[L^{1/2}(d^s)(x)=c_1\left\{\log^-|x-x_0|+c_2\chi_\Omega(x)\right\}\sqrt{\mathcal A(\nu(x_0))} + h_1(x),\]
where $h_1$ is $C^{\alpha}(\R^n)$, and $\log^-t=\min\{\log t,0\}$.
\end{prop}

To prove these results it is important to recall that, by Lemma \ref{lem-Fourier}, we have
\[L^{1/2}w(x)=\int_{\R^n}\bigl(w(x)-w(x+y)\bigr)\frac{b(y/|y|)}{|y|^{n+s}}\,dy\]
for some $b\in C^\infty(S^{n-1})$.

In the proof of Proposition \ref{proplaps2} we will also use the product rule
\[L^{1/2}(w_1w_2)=w_1L^{1/2}w_2+w_2L^{1/2}w_1-I(w_1,w_2),\]
where
\begin{equation}\label{Is}
I(w_1,w_2)(x)= \int_{\R^n} \bigl(w_1(x)-w_1(x+y)\bigr)\bigl(w_2(x)-w_2(x+y)\bigr)\frac{b(y/|y|)}{|y|^{n+s}}\,dy.
\end{equation}
The next lemma will lead to a H\"older bound for $I(d^s,v)$.

\begin{lem}\label{lem-boundI2}
Let $\Omega$ be a bounded domain, and $I$ be given by \eqref{Is}.
Then, for each $\alpha\in(0,1)$,
\begin{equation}\label{eq:bound-I2}
\|I(d^s,w)\|_{C^{\alpha/2}(\R^n)}\le C[w]_{C^\alpha(\R^n)},
\end{equation}
where the constant $C$ depends only on $n$, $s$, and $\alpha$.
\end{lem}

\begin{proof}
Let $x_1, x_2\in \R^n$.
Then,
\[|I(d^s,w)(x_1)-I(d^s,w)(x_2)| \le J_1 + J_2,\]
where
\[J_1= \int_{\R^n} \bigl|w(x_1)-w(x_1+y)-w(x_2)+w(x_2+y)\bigr|
\bigl|d^s(x_1)-d^s(x_1+y)\bigr|\,\frac{C}{|y|^{n+s}}\,dy  \]
and
\[J_2= \int_{\R^n} \bigl|w(x_2)-w(x_2+y)\bigr|
\bigl|d^s(x_1)-d^s(x_1+y)-d^s(x_2)+d^s(x_2+y)\bigr|\,\frac{C}{|y|^{n+s}}\,dy \,.\]
Using that $\|d^s\|_{C^{s}(\R^n)}\le 1$ and ${\rm supp}\,d^s = \overline\Omega$,
\[\begin{split}
J_1&\le \int_{\R^n}  \bigl|w(x_1)-w(x_1+y)-w(x_2)+w(x_2+y)\bigr|\min\{|y|^s,({\rm diam}\, \Omega)^s\} \frac{C}{|y|^{n+s}}\,dy \\
&\le  C\int_{\R^n}  [w]_{C^\alpha(\R^n)}|x_1-x_2|^{\alpha/2}|y|^{\alpha/2}\min\{|y|^s,1\} \frac{C}{|y|^{n+s}}\,dy \\
&\le  C |x_1-x_2|^{\alpha/2}[w]_{C^\alpha(\R^n)}\,.
\end{split}\]
Analogously,
\[J_2\le  C |x_1-x_2|^{\alpha/2}[w]_{C^\alpha(\R^n)}\,.\]

Finally, the bound for $\|I(d^s,w)\|_{L^\infty(\R^n)}$ is obtained with a similar argument, and hence \eqref{eq:bound-I2} follows.
\end{proof}

The following lemma, which is the analog of Lemma 4.3 in \cite{RS-Dir}, will be used in the proof of Proposition \ref{proplaps2} below (with $w$ replaced by $v$) and also in the next section (with $w$ replaced by $u$).

Recall that the norms $\|w\|_{\beta;\Omega}^{(\sigma)}$ were defined in the previous section.

\begin{lem}\label{lem-s-derivatives}
Let $\Omega$ be a bounded domain and $\alpha$ and $\beta$ be such that $0<\alpha\le s<\beta$ and $\beta-s$ is not an integer.
Let $k$ be an integer such that $\beta= k+\beta'$ with $\beta'\in(0,1]$.
Then,
\begin{equation}\label{eq:bound-lap-s/2}
[L^{1/2}w]_{\beta-s;\Omega}^{(s-\alpha)}\le C\bigl( \|w\|_{C^\alpha(\R^n)}+ \|w\|_{\beta;\Omega}^{(-\alpha)}\bigl)
\end{equation}
for all $w$ with finite right hand side.
The constant $C$ depends only on $n$, $s$, $\alpha$, and $\beta$ (but not on $\theta$).
\end{lem}

\begin{proof}
The proof is exactly the same as the one of Lemma 4.3 in \cite{RS-Dir}.
The only important point in the proof is that the kernel $b(y/|y|)$ is a $C^{\beta-s}$ function on the unit sphere -- which is the case here since $b\in C^\infty(S^{n-1})$.
\end{proof}

Next we give the:

\begin{proof}[Proof of Proposition \ref{proplaps2}]
Since $Lu\in L^\infty(\Omega)$, then $u/d^s|_\Omega$ is $C^{\alpha}(\overline\Omega)$ for some $\alpha\in(0,s)$.
Thus, we may define $v$ as a $C^\alpha(\R^n)$ extension of $u/d^s|_\Omega$.

Then, we have that
\[L^{1/2}u(x)=v(x)L^{1/2}d^s(x)+d^s(x)L^{1/2}v(x)-I(v,d^s),\]
where
\[I(v,d^s)= \int_{\R^n} \bigl(v(x)-v(x+y)\bigr)\bigl(d^s(x)-d^s(x+y)\bigr)\frac{b(y/|y|)}{|y|^{n+s}}\,dy.\]
This equality is valid in all of $\R^n$ because $d^s\equiv0$ in $\R^n\backslash \Omega$ and $v\in C^{\alpha+s}$ inside $\Omega$ -- by Proposition \ref{intregquotient}.
Thus, we only have to see that the terms $d^sL^{1/2}v$ and $I(v,d^s)$ belong to $C^{\alpha}(\R^n)$.

For the first one we combine Proposition \ref{intregquotient} with $\beta=s+\alpha$ and Lemma \ref{lem-s-derivatives}.
We obtain
\begin{equation}\label{paloma}
\|L^{1/2}v\|_{\alpha;\Omega}^{(s-\alpha)} \le C,
\end{equation}
and this immediately yields $d^sL^{1/2}v \in C^{\alpha}(\R^n)$; see the proof of Proposition 3.1 in \cite{RS-Poh} for more details.

The second bound, that is,
\[\|I(v,d^s)\|_{C^{\alpha}(\R^n)}\leq C,\]
follows from Lemma \ref{lem-boundI2}.
\end{proof}

Let us now prove Proposition \ref{prop:Lap-s/2-delta-s}.
For it, we need some lemmas.

\begin{lem}\label{claim1}
Let $L$ be given by \eqref{L}-\eqref{ellipt} with $a\in C^\infty(S^{n-1})$.

Let $\eta$ be a $C^\infty_c(\R)$ with support in $(-2,2)$ and such that $\eta\equiv1$ in $[-1,1]$.
Let $\nu\in S^{n-1}$, and
\begin{equation}\label{phi}
\phi_{x_0}(x)=\bigl((x-x_0)\cdot\nu\bigr)_-^s\,\eta((x-x_0)\cdot\nu),
\end{equation}
where $z_-=\min\{z,0\}$.
Then, we have
\[L^{1/2} \phi (x) = c_1\bigl\{\log |(x-x_0)\cdot \nu| + c_2\chi_{(0,\infty)}(x)\bigr\}\sqrt{\mathcal A(\nu)} + h(x)\]
for $x\in B_{1/2}(x_0)$, where $h\in C^s(B_{1/2}(x_0))$.
\end{lem}

\begin{proof}
On the one hand, since $\phi_{x_0}$ is a 1-D function, then by Lemma 2.1 in \cite{RS-K} we have that
\[L^{1/2}\phi_{x_0}(x)=c_s\,\mathcal B(\nu)(-\Delta)^{s/2}_{\R}\phi((x-x_0)\cdot\nu),\]
where $\phi(t)=(t_-)^s\,\eta(t)$ and
\[\mathcal B(\nu)=\int_{S^{n-1}}|\nu\cdot \theta|^sb(\theta)d\theta.\]
Moreover, by Lemma \ref{lem-Fourier}, we have $\mathcal B(\nu)=c\sqrt{\mathcal A(\nu)}$ for some constant $c$.

On the other hand, by Lemma 3.7 in \cite{RS-Poh}, we have that
\[(-\Delta)^{s/2}_{\R}\phi(t)=c_1\bigl\{\log |t| + c_2\chi_{(0,\infty)}(t)\bigr\} + h_0(t),\]
with $h_0\in C^s$.
Thus, the result follows.
\end{proof}

\begin{rem}\label{remrho0}
Throughout the rest of the Section the quantity~$\rho_0>0$ will be a fixed constant, depending only on $\Omega$, such that every point on $\partial\Omega$ can be touched from both inside and outside $\Omega$ by balls of radius $\rho_0$.
\end{rem}

\begin{lem}\label{claim2}
Let $s\in (0,1)$, and $L$ be an operator of the form \eqref{L}-\eqref{ellipt}, with $a\in C^\infty(S^{n-1})$.
Let $\Omega$ be any bounded $C^{1,1}$ domain in $\R^n$, and let $\rho_0$ be given by Remark \ref{remrho0}.

Fix $x_0\in\partial\Omega$, and define $\phi_{x_0}$ as in \eqref{phi}, with $\nu=\nu(x_0)$ the outward unit normal to $\partial\Omega$ at $x_0$.
Let us consider the segment
\begin{equation}\label{Sx0}
S_{x_0}=\{x_0+t\nu,\ t\in(-\rho_0/2,\rho_0/2)\},
\end{equation}
where $\phi$ is given by \eqref{phi} and $\nu$ is the unit outward normal to $\partial\Omega$ at $x_0$.
Define also $w_{x_0}=d^s-\phi_{x_0}$.

Then, for all $x\in S_{x_0}$,
\[|L^{1/2} w_{x_0}(x)-L^{1/2} w_{x_0}(x_0)| \le C|x-x_0|^{s/2},\]
where $C$ depends only on $\Omega$ and $\rho_0$ (and not on $x_0$).
\end{lem}

\begin{proof}
We denote $w=w_{x_0}$ and $\delta(x)={\rm dist}(x,\partial\Omega)$.

Note that, along $S_{x_0}$, the distance to $\partial\Omega$ agrees with the distance to the tangent plane to $\partial\Omega$ at $x_0$.
That is, denoting $\delta_\pm=(\chi_\Omega-\chi_{\R^n\backslash \Omega})\delta$ and $\bar d_2(x)=-\nu\cdot(x-x_0)$, we have $\delta_\pm(x)=\bar d_2(x)$ for all $x\in S_{x_0}$.
Moreover, the gradients of these two functions also coincide on $S_{x_0}$, i.e., $\nabla\delta_\pm(x)=-\nu=\nabla \bar d_2(x)$ for all $x\in S_{x_0}$.

Therefore, for all $x\in S_{x_0}$ and $y\in B_{\rho_0/2}(0)$, we have
\[|\delta_{\pm}(x+y)-\bar d_2(x+y)|\le C|y|^2\]
for some $C$ depending only on $\rho_0$.
Thus, for all $x\in S_{x_0}$ and $y\in B_{\rho_0/2}(0)$,
\begin{equation}\label{w(x+y)}
|w(x+y)|= |(\delta_{\pm}(x+y))_+^s-(\bar d_2(x+y))_+^s| \le C|y|^{2s},
\end{equation}
where $C$ is a constant depending on $\Omega$ and $s$.

On the other hand, since $w\in C^s(\R^n)$, then
\begin{equation}\label{w(x+y)2}
|w(x+y)-w(x_0+y)|\leq C|x-x_0|^s.
\end{equation}

Finally, let $\rho<\rho_0/2$ to be chosen later.
For each $x\in S_{x_0}$ we have
\[\begin{split}
|L^{1/2} &w(x)-L^{1/2} w(x_0)|
\le C\int_{\R^n} |w(x+y)-w(x_0+y)| \frac{C}{|y|^{n+s}}\,dy\\
&\le C\int_{B_\rho} |w(x+y)-w(x_0+y)| \frac{C}{|y|^{n+s}}\,dy\\
&\hspace{30mm}+C\int_{\R^n\setminus B_\rho} |w(x+y)-w(x_0+y)| \frac{C}{|y|^{n+s}}\,dy\\
&\le C\int_{B_\rho} |y|^{2s} \frac{C}{|y|^{n+s}}\,dy+
C\int_{\R\setminus B_\rho} |x-x_0|^s \frac{C}{|y|^{n+s}}\,dy\\
&= C(\rho^s +|x-x_0|^s\rho^{-s})\,,
\end{split}\]
where we have used \eqref{w(x+y)} and \eqref{w(x+y)2}.
Taking $\rho=|x-x_0|^{1/2}$ the lemma is proved.
\end{proof}

Finally, we give the proof of Proposition \ref{prop:Lap-s/2-delta-s}.

\begin{proof}[Proof of Proposition \ref{prop:Lap-s/2-delta-s}]
Let $\rho_0$ be given by Remark \ref{remrho0}, and
$$ U=\{x\in \R^n\,:\, \textrm{dist}(x,\partial\Omega)<\rho_0\}.$$
For each $x\in U$, let $x^*\in\partial\Omega$ be the unique point such that $|x-x^*|=\textrm{dist}(x,\partial\Omega)$.

Define
\[h_0(x)= L^{1/2}d^s(x)-c_1\left\{\log^-|x-x^*|+c_2\chi_\Omega(x)\right\}\sqrt{\mathcal A(\nu(x^*))}.\]
We claim that $h_0\in C^{\alpha}(U)$.

Indeed, we show next that we have
\begin{itemize}
\item[(i)] $h_0$ is locally Lipschitz in $U$ and
\[|\nabla h_0(x)| \le K|x-x^*|^{-M}\ \mbox{ in }\ U\]
for some $M>0$.
\item[(ii)] There exists $\alpha>0$ such that
\[|h_0(x)-h_0(x^*)|\le K|x-x^*|^\alpha \ \mbox{ in }\ U.\]
\end{itemize}

Then, (i) and (ii) yield that
\[\|h_0\|_{C^{\gamma}(\R^n)}\le CK\]
for some $\gamma>0$; see for example Claim 3.10 in \cite{RS-Poh}.

Let us show first (ii).
On one hand, by Lemma \ref{claim1}, for all $x_0\in\partial\Omega$ and for all $x\in S_{x_0}$, where $S_{x_0}$ is defined by \eqref{Sx0}, we have
\[h_0(x) = L^{1/2}d^s(x)- L^{1/2} \phi_{x_0}(x) + \tilde h(x),\]
where $\tilde h$ is the $C^s$ function from Lemma \ref{claim1}.
Hence, using Lemma \ref{claim2}, we find
\[|h_0(x)-h_0(x_0)|\le C|x-x_0|^{s/2}\quad \mbox{for all }x\in S_{x_0}\]
for some constant independent of $x_0$.

Recall that for all $x\in S_{x_0}$ we have $x^*=x_0$, where $x^*$ is the unique point on $\partial \Omega$ satisfying $\delta(x)= |x-x^*|$.
Hence, (ii) follows.

Let us now show (i).
Observe that $d^s\equiv0$ in $\R^n\backslash\Omega$, $|\nabla d^s|\leq Cd^{s-1}$ in $\Omega$, and $|D^2d^s|\leq Cd^{s-2}$ in $U$.
Then, letting $r=\textrm{dist}(x,\partial\Omega)/2$, we have
\[\begin{split}
|\nabla L^{1/2}d^s(x)| &\le
C\int_{\R^n} |\nabla d^s(x) - \nabla d^s(x+y)| |y|^{-n-s}\,dy\\
&\le C\int_{B_r} \frac{C r^{s-2}|y|\,dy}{|y|^{n+s}}+C\int_{\R\setminus B_r} \left(\frac{|\nabla d^s(x)|}{|y|^{n+s}}+ \frac{|\nabla d^s(x+y)|}{|y|^{1+s}}\right)dy\\
&\leq \frac Cr+\frac Cr+C\int_{\R^n\setminus B_r} |d(x+y)|^{s-1}\frac{dy}{|y|^{n+s}}.
\end{split}\]
Now, by Lemma 4.2 in \cite{RV} (with $\bar s$ and $\bar \alpha$ therein replaced by $s/2$ and $1-s/2$ here) we have that
\[\int_{\R^n\setminus B_r} |d(x+y)|^{s-1}\frac{dy}{|y|^{n+s}}\leq \frac{C}{r},\]
and thus we get
\[|\nabla L^{1/2}d^s(x)|\leq C|x-x^*|^{-1}.\]
This yields (i).

Thus, we have proved that $h_0\in C^\gamma(U)$ for some $\gamma>0$.

To finish the proof, we only have to notice that the function $|x-x^*|/|x-x_0|$ is Lipschitz in $\mathcal C_{x_0}\cap B_{1/2}(x_0)$ and bounded by below by a positive constant, so that
\[\log^-|x-x^*|-\log^-|x-x_0|\]
is Lipschitz in $\mathcal C_{x_0}\cap B_{1/2}(x_0)$.
Moreover, $\sqrt{\mathcal A(\nu(x^*)}-\sqrt{\mathcal A(\nu(x_0)}$ is also Lipschitz in $\mathcal C_{x_0}\cap B_{1/2}(x_0)$ and vanishes at $x=x_0$.
Thus, the function
\[\left\{\log^-|x-x^*|+c_2\chi_\Omega(x)\right\}\sqrt{\mathcal A(\nu(x^*))}-\left\{\log^-|x-x_0|+c_2\chi_\Omega(x)\right\}\sqrt{\mathcal A(\nu(x_0))}\]
is H\"older continuous in $\mathcal C_{x_0}\cap B_{1/2}(x_0)$.

This implies that
\[h(x)=L^{1/2}d^s(x)-c_1\left\{\log^-|x-x_0|+c_2\chi_\Omega(x)\right\}\sqrt{\mathcal A(\nu(x_0))}\]
is $C^\alpha$ in $\mathcal C_{x_0}\cap B_{1/2}(x_0)$, as desired.
\end{proof}

To end this section, we give the

\begin{proof}[Proof of Proposition \ref{thlaps/2}]
By Propositions \ref{proplaps2} and \ref{prop:Lap-s/2-delta-s}, we have that
\[L^{1/2}u(x)= c_1\bigl\{\log^-|x-x_0|+c_2\chi_{\Omega}(x)\bigr\}\sqrt{\mathcal A(\nu(x_0))}v(x)+h_1(x)\]
for some function $h_1\in C^\alpha(\mathcal C_{x_0}\cap B_\rho(x_0))$.

Thus, the result follows by taking into account that $v$ is $C^\alpha$ and that $v(x_0)=(u/d^s)(x_0)$.
\end{proof}

\section{Proof of the results in star-shaped domains}
\label{sec2}

In this section we prove Proposition \ref{intpartsB} for strictly star-shaped domains.
Recall that $\Omega$ is said to be strictly star-shaped if, for some $z_0\in\R^n$,
\begin{equation}\label{starshaped}
(x-z_0)\cdot\nu\geq c>0\qquad \textrm{for all}\ \ x\in\partial\Omega
\end{equation}
for some $c>0$.
The result for general $C^{1,1}$ domains will be a consequence of this strictly star-shaped case and will be proved in Section \ref{sec8}.

Before proving Proposition \ref{intpartsB} we state an essential ingredient in the proof of this result.
It is a fine 1-D computation that we did in \cite{RS-Poh}.

\begin{prop}[\cite{RS-Poh}]\label{propoperador}
Let $A$ and $B$ be real numbers, and
\[\varphi(t)=A\log^-|t-1|+B\chi_{[0,1]}(t)+h(t),\]
where $\log^-t=\min\{\log t,0\}$ and $h$ is a function satisfying, for some constants $\beta$ and $\gamma$ in $(0,1)$, and $C_0>0$, the following conditions:
\begin{itemize}
\item[(i)] $\|h\|_{C^{\beta}([0,\infty))}\leq C_0$.
\item[(ii)] For all $\beta\in[\gamma,1+\gamma]$
\[\|h\|_{C^{\beta}((0,1-\rho)\cup(1+\rho,2))}\leq C_0 \rho^{-\beta}\qquad \textrm{for all}\ \  \rho\in(0,1).\]
\item[(iii)] $|h'(t)|\leq C_0 t^{-2-\gamma}$ and $|h''(t)|\leq C_0 t^{-3-\gamma}$ for all $t>2$.
\end{itemize}
Then,
\[-\left.\frac{d}{d\lambda}\right|_{\lambda=1^+}\int_{0}^{\infty} \varphi\left(\lambda t\right)\varphi\left(\frac{t}{\lambda}\right)dt=A^2\pi^2+B^2.\]

Moreover, the limit defining this derivative is uniform among functions $\varphi$ satisfying (i)-(ii)-(iii) with given constants $C_0$, $\beta$, and $\gamma$.
\end{prop}

We can give now the

\begin{proof}[Proof of Proposition \ref{intpartsB} for strictly star-shaped domains]
By the argument in \cite[Section 2]{RS-Poh}, we may assume without loss of generality that $\Omega$ is strictly star-shaped with respect to the origin, that is, $z_0=0$ in \eqref{starshaped}.

We start with the identity
\begin{equation}\label{first}
\int_\Omega(x\cdot \nabla u) Lu\, dx=\left.\frac{d}{d\lambda}\right|_{\lambda=1^+}\int_{\R^n} u_\lambda Lu\, dx,
\end{equation}
where $u_\lambda(x)=u(\lambda x)$
and~$\left.\frac{d}{d\lambda}\right|_{\lambda=1^+}$ is the derivative from the right side at $\lambda=1$.
At a formal level, formula~\eqref{first} follows by taking
derivatives under the integral sign; rigorously, this can be
justified
using the bounds $|Lu|\leq C$ and $|\nabla u|\leq Cd^{s-1}$ in $\Omega$ and the fact that $u_\lambda\equiv0$ in $\R^n\setminus\Omega$ for $\lambda>1$.

Thus, as in \cite{RS-Poh}, integrating by parts and using the change of variables $y=\sqrt{\lambda}x$, we find
\[\int_{\R^n}u_{\lambda}Lu\,dx=\lambda^{\frac{2s-n}{2}}
\int_{\R^n}w_{\sqrt{\lambda}}w_{1/\sqrt{\lambda}}dy,\]
where
\[w(x)=L^{1/2}u(x),\quad \textrm{ and }\quad w_\lambda(x)=w(\lambda x).\]
This leads to
\begin{eqnarray}
\int_\Omega(\nabla u\cdot x)Lu\, dx&=&\left.\frac{d}{d\lambda}\right|_{\lambda=1^+}
\left\{\lambda^{\frac{2s-n}{2}}
\int_{\R^n}w_{\sqrt{\lambda}}w_{1/\sqrt{\lambda}}dy\right\}\nonumber \\
&=&\frac{2s-n}{2}\int_{\R^n}|w|^2dx \\& & \qquad +\left.\frac{d}{d\lambda}\right|_{\lambda=1^+}
\int_{\R^n}w_{\sqrt{\lambda}}w_{1/\sqrt{\lambda}}dy\nonumber \\
&=&\frac{2s-n}{2}\int_{\Omega}uLu\, dx
+ \frac{1}{2}\left.\frac{d}{d\lambda}\right|_{\lambda=1^+}\int_{\R^n}w_{\lambda}w_{1/\lambda}dy.
\label{igualtatpaluego}
\end{eqnarray}
Hence, we have to prove that
\begin{equation}\label{derivadaIlambda}
-\left.\frac{d}{d\lambda}\right|_{\lambda=1^+}I_\lambda
=\Gamma(1+s)^2\int_{\partial\Omega}\mathcal A(\nu)\left(\frac{u}{d^s}\right)^2(x\cdot\nu)\,d\sigma,
\end{equation}
where
\begin{equation}\label{Ilambda}
I_\lambda=\int_{\R^n}w_{\lambda}w_{1/\lambda}dy.
\end{equation}
We write the integral \eqref{derivadaIlambda} in coordinates $(t,x_0)\in (0,\infty)\times \partial\Omega$, where each $y\in \R^n$ is written as $y=tx_0$.
We find
\begin{equation}\label{astu}
\left.\frac{d}{d\lambda}\right|_{\lambda=1^+}I_\lambda
=\left.\frac{d}{d\lambda}\right|_{\lambda=1^+}\int_{\partial\Omega}(x\cdot \nu)d\sigma(x)\int_0^\infty t^{n-1}w(\lambda tx)w\left(\frac{tx}{\lambda}\right)dt.
\end{equation}
Fix now $x_0\in \partial\Omega$, and define
\[\varphi(t)=t^{\frac{n-1}{2}}w\left(t x_0\right)=t^{\frac{n-1}{2}}L^{1/2}u(t x_0).\]
By Theorem \ref{thlaps/2}, we have
\[\varphi(t)=t^{\frac{n-1}{2}}\sqrt{\mathcal A(\nu)} c_1\bigl\{\log^-|t-1|+c_2\chi_{(0,1)}(t)\bigr\}\left(\frac{u}{d^s}\right)(x_0)+h_1(t)\]
in $[0,\infty)$, where $h_1$ is a $C^{\gamma}([0,\infty))$ function.

Thus, this yields
\[\varphi(t)=\sqrt{\mathcal A(\nu)} c_1\bigl\{\log^-|t-1|+c_2\chi_{(0,1)}(t)\bigr\}\left(\frac{u}{d^s}\right)(x_0)+h(t)\]
in $[0,\infty)$, where $h\in C^{\gamma}([0,\infty))$.

We want to apply now Proposition \ref{propoperador} to this function $\varphi(t)$.
For this, we have to check that (ii), and (iii) hold -- we already checked (i).

To check (ii), we just apply Lemma \ref{lem-s-derivatives} with $w=u$, $\beta\in(0,1+s)$, and $\alpha=s$.
We find that $\varphi$ satisfies the bound in (ii), and thus $h$ also satisfies the same bound.

To check (iii), we notice that for $x\in \R^n\setminus (2\Omega)$ we have
\[L^{1/2}u(x)=-\int_{\Omega}u(y)K(x-y)dy,\]
where $K(y)=b(y/|y|)|y|^{-n-s}$.
Since $b\in C^\infty(S^{n-1})$, differentiating under the integral sign one gets
\[|\nabla L^{1/2}u(x)|\leq C|x|^{-n-s-1}\quad\textrm{and}\quad |D^2L^{1/2}u(x)|\leq C|x|^{-n-s-2}.\]
And this yields (iii).

Therefore, we can apply Proposition \ref{propoperador} to find that, for each $x_0\in \partial\Omega$,
\[\left.\frac{d}{d\lambda}\right|_{\lambda=1^+}\int_0^\infty t^{n-1}w(\lambda tx)w\left(\frac{tx}{\lambda}\right)dt=c\,\mathcal A(\nu(x_0))\left(\frac{u}{d^s}\right)^2(x_0)\]
for some constant $c$.

Finally, by uniform convergence on $x_0$ of the limit, and by \eqref{astu}, this leads to
\[\left.\frac{d}{d\lambda}\right|_{\lambda=1^+}I_\lambda
=c\int_{\partial\Omega}\bigl(x_0\cdot\nu\bigr)\mathcal A(\nu)\left(\frac{u}{d^s}\right)^2dx_0,\]
which is exactly what we wanted to prove.
\end{proof}

\section{Non-star-shaped domains and proof of Theorem \ref{thpoh}} \label{sec8}

In this section we prove Proposition \ref{intpartsB} for general $C^{1,1}$ domains.

The key idea, as in \cite{RS-Poh}, is that every $C^{1,1}$ domain is locally star-shaped, in the sense that its intersection with any small ball is star-shaped with respect to some point.
To exploit this, we use a partition of unity to split the function $u$ into a set of functions $u_1$, ..., $u_m$, each one with support in a small ball.
Using this, we will prove a bilinear version of the identity, namely
\begin{equation}\label{bilinear}\begin{split}
\int_\Omega(x\cdot\nabla &u_1)Lu_2\, dx+\int_\Omega(x\cdot\nabla u_2)Lu_1\,dx=\frac{2s-n}{2}\int_\Omega u_1Lu_2\, dx+\\
&+\frac{2s-n}{2}\int_{\Omega}u_2Lu_1\, dx-\Gamma(1+s)^2\int_{\partial\Omega}\mathcal A(\nu)\frac{u_1}{d^{s}}\frac{u_2}{d^{s}}(x\cdot\nu)\, d\sigma.
\end{split}\end{equation}

The following lemma states that this bilinear identity holds whenever the two functions $u_1$ and $u_2$ have disjoint compact supports.
In this case, the last term in the previous identity equals 0, and since $L u_i$ is evaluated only outside the support of $u_i$, we only need to require $\nabla u_i\in L^1(\R^n)$.

\begin{lem}\label{duesboles} Let $u_1$ and $u_2$ be $W^{1,1}(\R^n)$ functions with disjoint compact supports $K_1$ and $K_2$.
Then,
\[\begin{split}\int_{K_1}(x\cdot\nabla u_1)Lu_2\, dx&+\int_{K_2}(x\cdot\nabla u_2)Lu_1\,dx=\\
&=\frac{2s-n}{2}\int_{K_1}u_1Lu_2\, dx+\frac{2s-n}{2}\int_{K_2}u_2Lu_1\, dx.\end{split}\]
\end{lem}

\begin{proof}
Notice first that
\begin{equation}\label{agur}
Lw(x)=c_s\int_{S^{n-1}}(-\partial_{\theta\theta})^sw(x)d\mu(\theta),
\end{equation}
see e.g. formula~(2.2) and Lemma 2.1 in \cite{RS-K}.

We claim that, for each $\theta\in S^{n-1}$,
\begin{equation}\label{idcabre-1}
(-\partial_{\theta\theta})^s(x\cdot\nabla u_i)=x\cdot\nabla(-\partial_{\theta\theta})^su_i+2s(-\partial_{\theta\theta})^su_i\qquad \mbox{in}\ \  \R^n\backslash K_i.
\end{equation}
Indeed, fix $\theta\in S^{n-1}$ and fix $x_0\in \{x+\tau\theta\,:\,\tau\in \R\}$.
Let $\tau_1$ be such that $x_0+\tau_1\theta=x$.
Then, using that $u_i\equiv0$ in $\R^n\setminus K_i$, for each $x\in \R^n\backslash K_i$ we have
\[\begin{split}
(-\partial_{\theta\theta})^s(x\cdot\nabla u_i)(x)&=
c_{1,s}\int_{x_0+\tau\theta \in K_i}\frac{-(x_0+\tau\theta)\cdot\nabla u_i(x_0+\tau\theta)}{|\tau-\tau_1|^{1+2s}}\,d\tau\\
&=c_{1,s}\int_{x_0+\tau\theta \in K_i}\frac{(\tau-\tau_1)\theta\cdot\nabla u_i(x_0+\tau\theta)}{|\tau-\tau_1|^{1+2s}}\,d\tau\\
&\qquad\qquad\qquad\qquad+c_{1,s}\int_{x_0+\tau\theta \in K_i}\frac{-(x_0+\tau_1\theta)\cdot\nabla u_i(x_0+\tau\theta)}{|\tau-\tau_1|^{1+2s}}\,d\tau\\
&=c_{1,s}\int_{x_0+\tau\theta \in K_i}\partial_\tau\left(\frac{\tau_1-\tau}{|\tau-\tau_1|^{1+2s}}\right)u_i(y)d\tau+x\cdot(-\partial_{\theta\theta})^s\nabla u_i(x)\\
&=c_{1,s}\int_{x_0+\tau\theta \in K_i}\frac{-2s}{|\tau-\tau_1|^{1+2s}}\,u_i(y)d\tau+x\cdot\nabla(-\partial_{\theta\theta})^su_i(x)\\
&=2s(-\partial_{\theta\theta})^su_i(x)+x\cdot\nabla(-\partial_{\theta\theta})^su_i(x),
\end{split}\]
as claimed.

Therefore, using \eqref{idcabre-1} and \eqref{agur}, we find
\begin{equation}\label{idcabre}
L(x\cdot\nabla u_i)=x\cdot\nabla Lu_i+2s\,Lu_i\qquad \mbox{in}\ \  \R^n\backslash K_i.
\end{equation}

We also note that for all functions $w_1$ and $w_2$ in $L^1(\R^n)$ with disjoint compact supports $W_1$ and $W_2$, it holds the integration by parts formula
\begin{equation}\label{above}
\int_{W_1}w_1Lw_2=\int_{W_1}\int_{W_2}\frac{-w_1(x)w_2(y)}{|x-y|^{n+2s}}a\left(\frac{x-y}{|x-y|}\right)dy\,dx
=\int_{W_2}w_2Lw_1.
\end{equation}

Now, integrating by parts,
\[\int_{K_1}(x\cdot \nabla u_1)Lu_2= -n\int_{K_1}u_1Lu_2-\int_{K_1}u_1x\cdot\nabla Lu_2.\]
Next we apply \eqref{idcabre} and \eqref{above} to $w_1=u_1$ and $w_2=x\cdot \nabla u_2$.
We obtain
\[\begin{split}
\int_{K_1}u_1x\cdot\nabla Lu_2&=\int_{K_1}u_1L(x\cdot \nabla u_2)-2s\int_{K_1}u_1Lu_2\\
&=\int_{K_2}Lu_1(x\cdot \nabla u_2)-2s\int_{K_1}u_1Lu_2.
\end{split}\]
Hence,
\[\int_{K_1}(x\cdot \nabla u_1)Lu_2=-\int_{K_2}Lu_1(x\cdot \nabla u_2)+(2s-n)\int_{K_1}u_1Lu_2.\]

Finally, again by the integration by parts formula \eqref{above} we find
\[\int_{K_1}u_1Lu_2=\frac12\int_{K_1}u_1Lu_2+\frac12\int_{K_2}u_2Lu_1,\]
and the lemma follows.
\end{proof}

The second lemma states that the bilinear identity \eqref{bilinear} holds whenever the two functions $u_1$ and $u_2$ have compact supports in a ball $B$ such that $\Omega\cap B$ is star-shaped with respect to some point $z_0$ in $\Omega\cap B$.

\begin{lem}\label{unabola}
Let $\Omega$ be a bounded $C^{1,1}$ domain, and let $B$ be a ball in $\R^n$.
Assume that there exists $z_0\in \Omega\cap B$ such that
\[(x-z_0)\cdot\nu(x)>0\qquad\mbox{for all}\ x\in\partial\Omega\cap \overline B.\]
Let $u$ be a function satisfying the hypothesis of Proposition \ref{intpartsB}, and let $u_1=u\eta_1$ and $u_2=u\eta_2$, where $\eta_i\in C^\infty_c(B)$, $i=1,2$.
Then, the following identity holds
\[\int_B(x\cdot\nabla u_1)Lu_2\, dx+\int_B(x\cdot\nabla u_2)Lu_1\,dx=\frac{2s-n}{2}\int_Bu_1Lu_2\, dx+\]
\[+\frac{2s-n}{2}\int_{B}u_2Lu_1\, dx
-\Gamma(1+s)^2\int_{\partial\Omega\cap B}\mathcal A(\nu)\frac{u_1}{d^{s}}\frac{u_2}{d^{s}}(x\cdot\nu)\, d\sigma.\]
\end{lem}

\begin{proof}
The proof is exactly the same as Lemma 5.2 in \cite{RS-Poh}.
One only has to check that for all $\eta\in C^\infty_c(B)$, and letting $\tilde u=u\eta$, then the function $\tilde u$ satisfies the hypotheses of Proposition \ref{intpartsB}.

Hypotheses (a) and (b) are immediate to check, since $\eta$ is smooth.
So, we only have to check that $L\tilde u$ is bounded.
But
\[L(u\eta)=\eta Lu+uL\eta-I_L(u,\eta),\]
where $I_L$ is given by \eqref{I_L}.
The first term is bounded because $Lu$ is bounded.
The second term is bounded since $\eta\in C^\infty_c(B)$.
The third term is bounded because $u\in C^s(\R^n)$ and $\eta\in \textrm{Lip}(\R^n)$.
Thus, the lemma is proved.
\end{proof}

We now give the

\begin{proof}[Proof of Proposition \ref{intpartsB}]\label{intparts:page}
As in \cite{RS-Poh}, the result follows from Lemmas \ref{unabola} and \ref{duesboles}.
We omit the details of this proof because it is exactly the same as in \cite{RS-Poh}.
\end{proof}

Hence, recalling the result in Section \ref{approx}, Proposition \ref{intparts} is proved.

Finally, as in \cite{RS-Poh}, the other results follow from Proposition \ref{intparts}.

\begin{proof}[Proof of Theorem \ref{thpoh}]
The first identity follows immediately from Proposition \ref{intparts} and the results in \cite{RV}.
The second identity follows from the first one by applying it with two different origins; see \cite{RS-Poh} for more details.
\end{proof}

\begin{proof}[Proof of Corollary \ref{corpoh}]
The result follows immediately from the first identity in Theorem \ref{thpoh}.
\end{proof}

\begin{proof}[Proof of Corollary \ref{corintparts}]
Applying Proposition \ref{intparts} with two different origins, we find that
\[\int_\Omega w_{x_i}Lw\,dx=-\frac{\Gamma(1+s)^2}{2}\int_{\partial\Omega}\mathcal A(\nu)\left(\frac{w}{d^{s}}\right)^2\nu_i\,d\sigma\]
whenever $w$ satisfies the hypotheses of the Proposition.
Then, the result follows by applying this identity with $w=u+v$ and $w=u-v$, and subtracting the two identities.
\end{proof}

\section{Applications of the identities} \label{sec9}

We give here some consequences of our identities.

A typical application of Pohozaev-type identities is the nonexistence of solutions to $Lu=u^p$, with $p\geq \frac{n+2s}{n-2s}$.
For supercritical powers $p>\frac{n+2s}{n-2s}$, the nonexistence of bounded solutions was already known, since it follows from the results in
\cite{RS-nonexistence}.
For the critical nonlinearity $f(u)=u^{\frac{n+2s}{n-2s}}$, the nonexistence of \emph{bounded} positive solutions follows directly from
Corollary \ref{corpoh} (see \cite{RS-Poh}), and hence the nonexistence of all positive solutions follows combining this with the following result, which we prove here.

\begin{prop}\label{L^infty}
Let $\Omega$ be any bounded domain, and $f(x,u)$ be such that
\begin{equation}\label{9dsf}
|f(x,u)|\le C_0\left(1+|u|^{\frac{n+2s}{n-2s}}\right).
\end{equation}
Let $L$ be any operator of the form \eqref{L}-\eqref{ellipt}, and $u$ be any weak solution of \eqref{eq}.

Then
\begin{equation}\label{u-L^infty}
\|u\|_{L^\infty(\Omega)}\le C,
\end{equation}
for some $C>0$ depending only on $n$, $s$, $C_0$, ellipticity constants, and $\|u\|_{H^s_{\mu}(\R^n)}$.
\end{prop}

\begin{rem}
Here, we say that $u$ is a weak solution of \eqref{eq} if $u\equiv0$ in $\Omega^c$,
\[\|u\|_{H^s_\mu(\R^n)}^2:=\int_{\R^n}\int_{S^{n-1}}\int_{-\infty}^\infty \bigl(u(x)-u(x+r\theta)\bigr)^2\frac{dr}{|r|^{1+2s}}\,d\mu(\theta)\,dx\]
is finite, and
\[\int_{\R^n}\int_{S^{n-1}}\int_{-\infty}^\infty \bigl(u(x)-u(x+r\theta)\bigr)^2\frac{dr}{|r|^{1+2s}}\,d\mu(\theta)\,dx=\int_\Omega f(x,u)\eta\,dx\]
for all $\eta\in C^\infty_c(\Omega)$.

By Lemma \ref{lem-H^s} below, the norm $\|u\|_{H^s_\mu(\R^n)}$ is equivalent to $\|u\|_{H^s(\R^n)}$, so that weak solutions belong to $H^s$.
\end{rem}

Another consequence of Corollary \ref{corpoh} and Proposition \ref{L^infty} is the following unique continuation principle.
Recall that a nonlinearity $f(u)$ is said to be subcritical if
\begin{equation}\label{subcritical}
t\,f(t)<\frac{n-2s}{2n}\int_0^t f
\end{equation}
for all $t\neq0$.

\begin{cor}\label{cor-uniquecont}
Let $s\in(0,1)$, and assume that $L$ and $\Omega$ satisfy \eqref{A1}.

Let $f$ be any locally Lipschitz function, and $u$ be any weak solution of \eqref{eq2}.
Assume in addition that $f(u)$ is subcritical, in the sense that \eqref{subcritical} holds.

Then, $u$ is bounded in $\Omega$, $u/d^s$ is H\"older continuous up to the boundary, and the following unique continuation principle holds:
\[\left.\frac{u}{d^s}\right|_{\partial\Omega}\equiv0\quad \textrm{on}\quad \partial\Omega\qquad \Longrightarrow\qquad u\equiv0\quad \textrm{in}\quad \Omega.\]
Here, $u/d^s$ on $\partial\Omega$ has to be understood as a limit (as in Theorem \ref{thpoh}).
\end{cor}

We next prove Proposition \ref{L^infty} and Corollary \ref{cor-uniquecont}.

To establish Proposition \ref{L^infty}, we will need the following.

\begin{lem}\label{lem-H^s}
Let $L$ be any operator of the form \eqref{L-singular}-\eqref{ellipt-singular}.
Then,
\[c[u]_{H^s(\R^n)}^2\leq \int_{\R^n}\int_{S^{n-1}}\int_{-\infty}^\infty \bigl(u(x)-u(x+r\theta)\bigr)^2\frac{dr}{|r|^{1+2s}}\,d\mu(\theta)dx\leq C[u]_{H^s(\R^n)}^2,\]
where the constants $c$ and $C$ depend only on the ellipticity constants in \eqref{ellipt-singular}.
\end{lem}

\begin{proof}
The result follows by writing each of the terms in the Fourier side.
Indeed, since the symbol of $L$ is $A(\xi)$, and it satisfies
\[\lambda|\xi|^{2s}\leq A(\xi)\leq \Lambda|\xi|^{2s},\]
then we have
\[c\int_{\R^n}|\xi|^{2s}|\hat u|^2d\xi\leq \int_{\R^n}A(\xi)|\hat u|^2d\xi\leq C\int_{\R^n}|\xi|^{2s}|\hat u|^2d\xi,\]
as desired.
\end{proof}

We will also need the following result, established in \cite{FR}.

\begin{prop}[\cite{FR}] \label{prop-FR}
Let $\Omega\subset \R^n$ be any bounded domain, and $L$ any operator of the form \eqref{L}-\eqref{ellipt}.
Let $u$ be any weak solution of
\[\left\{ \begin{array}{rcll}
L u&=&g& \textrm{in }\Omega\\
u&=&0& \textrm{in } \R^n\setminus \Omega,
\end{array}\right.\]
Then,
\begin{enumerate}
\item[(i)] If $1<p<\frac{n}{2s}$,
   \[\|u\|_{L^q(\Omega)} \leq C\|g\|_{L^p(\Omega)},~~~q= \frac{np}{n-2ps}.\]
\item[(ii)] If $\frac{n}{2s} < p < \infty$,
   \[ \|u\|_{L^\infty(\Omega)} \leq C\|g\|_{L^p(\Omega)}.\]
\end{enumerate}
The constant $C$ depends only on $n$, $s$, $p$, $\Omega$ and ellipticity constants.
\end{prop}

The last ingredient for the proof of Proposition \ref{L^infty} is the following technical result.

\begin{lem}\label{useful-ineq}
Fix $T>0$ and $\beta\geq0$.
Then, for all real numbers $a,b$, we have
\[\bigl|a_T^\beta a - b_T^\beta b\bigr|^2\leq C(a-b)\bigl(a_T^{2\beta}a - b_T^{2\beta}b\bigr),\]
where $a_T=\min\{|a|,T\}$ and $b_T=\min\{|b|,T\}$.
The constant $C$ depends only on $\beta$.
\end{lem}

\begin{proof}
Let
\[f(z)=z\cdot\bigl(\min\{|z|,T\}\bigr)^\beta.\]
Then, we clearly have
\[|f(a)-f(b)|^2=\left(\int_a^bf'\right)^2\leq (a-b)\int_a^b (f')^2.\]
Also,
\[|f(a)-f(b)|^2=\bigl|a_T^\beta a - b_T^\beta b\bigr|^2,\]
so that we only have to show that
\begin{equation}\label{have-to}
 (a-b)\int_a^b (f')^2 \leq (a-b)\bigl(a_T^{2\beta}a - b_T^{2\beta}b\bigr).
\end{equation}
But
\[f'(z)=\left\{\begin{array}{ll} T^\beta & \textrm{if}\ |z|>T \\
(\beta+1)|z|^\beta & \textrm{if}\ |z|<T,
               \end{array}\right.\]
and therefore
\[\bigl(\min\{|z|,T\}\bigr)^\beta \leq f'(z)\leq (\beta+1)\bigl(\min\{|z|,T\}\bigr)^\beta.\]

Similarly, the function
\[g(z)=z\cdot\bigl(\min\{|z|,T\}\bigr)^{2\beta}\]
satisfies
\[\bigl(\min\{|z|,T\}\bigr)^{2\beta} \leq g'(z)\leq (\beta+1)\bigl(\min\{|z|,T\}\bigr)^{2\beta}.\]
Thus,
\[(a-b)\int_a^b(f')^2\leq (\beta+1)^2(a-b)\int_a^bg'=C(a-b)\bigl(g(a)-g(b)\bigr),\]
and this yields \eqref{have-to}.
\end{proof}

We give now the:

\begin{proof}[Proof of Proposition \ref{L^infty}]
We adapt a classical argument of Brezis-Kato for $-\Delta u=f(x,u)$ to the present context of nonlocal equations.

Fix $\beta\geq0$ and $T>1$, and let $u_T=\min\{|u|,T\}$.
By Lemma \ref{useful-ineq}, for all $x,y\in \R^n$,
\begin{equation}\label{ineq-useful}
\bigl|u(x)u_T^\beta(x)-u(y)u_T^\beta(y)\bigr|^2\leq C\bigl(u(x)-u(y)\bigr)\bigl(u(x)u_T^{2\beta}(x)-u(y)u_T^{2\beta}(y)\bigr).
\end{equation}
Hence, using \eqref{ineq-useful}, we find
\[\begin{split}
\int_{\R^n}\int_{\R^n}&\bigl|u(x)u_T^\beta(x)-u(y)u_T^\beta(y)\bigr|^2K(x-y)dx\,dy\\
&\leq C\int_{\R^n}\int_{\R^n}\bigl(u(x)-u(y)\bigr)\bigl(u(x)u_T^{2\beta}(x)-u(y)u_T^{2\beta}(y)\bigr)K(x-y)dx\,dy,\end{split}\]
where we denoted $K(y)=a(y/|y|)|y|^{-n-2s}$.

Moreover, using the equation \eqref{eq}, we also have
\[\int_{\R^n}\int_{\R^n}\bigl(u(x)-u(y)\bigr)\bigl(u(x)u_T^{2\beta}(x)-u(y)u_T^{2\beta}(y)\bigr)K(x-y)dx\,dy=\int_\Omega f(x,u)\,u\, u_T^{2\beta}dx.\]

Now, by \eqref{9dsf}, we have that
\[|f(x,u)|\leq \alpha(x)\bigl(1+|u|\bigr),\]
with
\[\alpha(x)=\frac{|f(x,u)|}{1+|u|}\leq C\bigl(1+|u|^{\frac{4s}{n-2s}}\bigr)\in L^{\frac{n}{2s}}(\Omega).\]
We have used that $u\in L^{\frac{2n}{n-2s}}(\Omega)$, since $u\in H^s(\R^n)$ by Lemma \ref{lem-H^s}.

Combining these facts, we find
\[\int_{\R^n}\int_{\R^n}\bigl|u(x)u_T^\beta(x)-u(y)u_T^\beta(y)\bigr|^2K(x-y)dx\,dy\leq C\int_\Omega\alpha(x)(1+|u|)^2u_T^{2\beta}dx,\]
and thus, using Lemma \ref{lem-H^s},
\[\bigl[uu_T^\beta\bigr]_{H^s(\R^n)}^2\leq C\int_\Omega\alpha(x)(1+|u|)^2u_T^{2\beta}dx.\]
Therefore, by the fractional Sobolev inequality,
\begin{equation}\label{step5}
\left(\int_\Omega |uu_T^\beta|^{\frac{2n}{n-2s}}dx\right)^{\frac{n-2s}{2n}}\leq C_1\int_\Omega\alpha(x)(1+|u|)^2u_T^{2\beta}dx.
\end{equation}

Assume that
\[\int_\Omega |u|^{2+2\beta}dx\leq C_2\]
for some $\beta\geq0$.
Then,
\[\begin{split}
\int_\Omega \alpha(x)|u|^2u_T^{2\beta}dx&\leq M_0\int_\Omega |u|^{2+2\beta}dx+\int_{\{\alpha(x)>M_0\}}\alpha(x)|u|^2u_T^{2\beta}dx\\
&\leq C_2M_0+\varepsilon(M_0)\left(\int_\Omega |uu_T^\beta|^{\frac{2n}{n-2s}}dx\right)^{\frac{n-2s}{2n}},
\end{split}\]
where
\[\varepsilon(M_0)=\left(\int_{\{\alpha(x)>M_0\}}|\alpha(x)|^{n/2s}dx\right)^{2s/n}\longrightarrow 0\]
as $M_0\rightarrow\infty$.
Also, note that we can deal with $\int_\Omega\alpha(x) u_T^{2\beta}dx$ in the analogue procedure.

Therefore, taking $M_0$ large enough so that $C_1\varepsilon(M_0)\leq 1/2$, we find
\[\left(\int_\Omega |uu_T^\beta|^{\frac{2n}{n-2s}}dx\right)^{\frac{n-2s}{2n}}\leq CC_2,\]
with $C$ independent of $T$.
Thus, letting $T\rightarrow\infty$, we obtain that
\[\int_\Omega |u|^{(2+2\beta)\frac{n}{n-2s}}dx\leq CC_2.\]
Hence, iterating $\beta_0=0$, $1+\beta_k=(1+\beta_{k-1})\frac{n}{n-2s}$ for $k\geq1$, we conclude that $u\in L^p(\Omega)$ for all $p<\infty$.

Finally, by Proposition \ref{prop-FR} and \eqref{9dsf}, this yields $u\in L^\infty(\Omega)$, as desired.
\end{proof}

\begin{rem}
Notice that Proposition \ref{L^infty} establishes the boundedness of solutions for critical and subcritical nonlinearities $|f(x,u)|\leq C\left(1+|u|^{\frac{n+2s}{n-2s}}\right)$ whenever the operator $L$ satisfies \eqref{L}-\eqref{ellipt}, but the assumption \eqref{ellipt} is only needed in order to apply Proposition \ref{prop-FR}.

For subcritical nonlinearities $|f(x,u)|\leq C(1+|u|^p)$, with $p<\frac{n+2s}{n-2s}$, the result in Proposition \ref{L^infty} could be proved by using the argument in \cite[Theorem 2.3]{DMPV}.
In this proof, only does not need to use Proposition \ref{prop-FR} but only Lemma \ref{lem-H^s}, and thus the result would be true for any operator \eqref{L-singular}-\eqref{ellipt-singular}.
\end{rem}

We can finally give the:

\begin{proof}[Proof of Corollary \ref{cor-uniquecont}]
First, since $f$ is locally Lipschitz and \eqref{subcritical} holds, then
\[|f(x,u)|\leq C\bigl(1+|u|^{\frac{n+2s}{n-2s}}\bigr).\]
Hence, by Proposition \ref{L^infty}, the solution $u$ is bounded, and by Theorem \ref{krylov} $u/d^s\in C^\alpha(\overline\Omega)$.

Assume that $u/d^s|_{\partial\Omega}\equiv0$ on $\partial\Omega$.
Then, by Corollary \ref{corpoh} we have
\[\int_\Omega\left\{F(u)-\frac{2n}{n-2s}\,u\,f(u)\right\}=0.\]
But since
\[F(t)-\frac{2n}{n-2s}\,t\,f(t)>0\]
whenever $t\neq0$, then we find that $u\equiv0$ in $\Omega$.
\end{proof}

\end{document}